\documentclass[a4paper,11pt]{scrartcl}
\usepackage{nccmath,amsthm,mathtools}
\usepackage[mathscr]{euscript}
\usepackage{tikz}
\usepackage{dsfont}
\usepackage{color}
\usepackage{subfig}
\usepackage{xcolor}

\usepackage[utopia]{mathdesign} %Change the math fonts
\usepackage[mathscr]{euscript}
\usepackage{microtype}

\numberwithin{equation}{section}

\definecolor{light}{gray}{.75}
\newcommand{\Argmin}{\mathrm{Argmin}}
\newcommand{\zero}{\mathbf{0}}
\newcommand{\R}{\mathds{R}}
\newcommand{\co}{\mathrm{co}}
\newcommand{\sign}{\mathrm{sign}}
\newcommand{\super}{\overline{\partial}}
\newcommand{\sub}{\underline{\partial}}
\renewcommand{\S}{\mathscr{S}} %Stable knots (their gradients)
\newcommand{\U}{\mathscr{U}} %Unstable knots (their gradients)
\newcommand{\A}{\mathscr{A}}
\newcommand{\B}{\mathscr{B}}
\newcommand{\C}{\mathscr{C}}
\newcommand{\PP}{\mathscr{P}}

\newtheorem{theorem}{Theorem}[section]
\newtheorem{proposition}{Proposition}[section]
\newtheorem{corollary}{Corollary}[section]

\newtheorem{lemma}{Lemma}[section]

\theoremstyle{definition}
\newtheorem{remark}{Remark}[section]
\newtheorem{definition}{Definition}[section]
\newtheorem{example}{Example}[section]

\def\clap#1{\hbox to 0pt{\hss#1\hss}}
\def\mathclap{\mathpalette\mathclapinternal}
\def\mathclapinternal#1#2{%
  \clap{$\mathsurround=0pt#1{#2}$}%
}
\def\mathrlap{\mathpalette\mathrlapinternal}
\def\mathrlapinternal#1#2{%
  \rlap{$\mathsurround=0pt#1{#2}$}% $
}

  \title{Characterization theorem for best polynomial spline approximation with free knots}
  \author{Nadezda Sukhorukova \and Julien Ugon}
\begin{document}
\maketitle
  \begin{abstract}
   In this paper, we derive a necessary condition for a best approximation by
   piecewise polynomial functions. We apply nonsmooth nonconvex analysis to
   obtain this result, which is also a necessary and sufficient condition for
   inf-stationarity in the sense of Demyanov-Rubinov. We start from identifying
   a special property of the knots. Then, using this property, we construct a
   characterization theorem for best free knots polynomial spline approximation,
   which is stronger than the existing characterisation results when only
   continuity is required.
  \end{abstract}
  %\textbf{keywords:} polynomial splines, Chebyshev approximation, free knots, quasidifferential, best approximation conditions\\

  \section{Introduction}\label{sec:introduction}

    The problem of approximating a continuous function by a piecewise
    polynomial (polynomial spline) has been studied for over four
    decades~\cite{Schumaker68}; yet, when the knots joining the polynomials are
    also variable, finding conditions for a best Chebyshev approximation
    remains an open problem~\cite[problem~1]{FreeKnotsOpenProblem96}. We derive
    a necessary optimality condition for a best approximation which is stronger
    than the existing ones~\cite{nurnberger,Nurnberger892} when only continuity
    is required .

    It is acknowledged in~\cite{FreeKnotsOpenProblem96} that the existing
    optimisation tools are not adapted to this problem, due to its nonconvex
    and nonsmooth nature. Therefore,  our motivation is to apply recently
    developed  nonsmooth optimisation tools~\cite{dr95,Dr00} not
    well-known outside of the optimisation community to improve the
    existing results.

    In this paper we are concentrating on necessary optimality conditions.
    These conditions are important for the development of an algorithm for
    constructing a best polynomial spline approximation, since these conditions
    can be used as stopping criteria. Most local optimisation methods can
    improve a solution that is not locally optimal. Therefore, it is especially
    important to confirm that an iterate is at least locally optimal or use it
    as a starting point for local optimisation.

    The paper is organised as follows. In section~\ref{sec:preliminaries} we
    introduce necessary definitions and relevant results from the area of
    polynomials spline approximation. The last subsection of this section
    provides necessary results from the theory of quasidifferentials, developed
    by Demyanov and Rubinov~\cite{dr95,Dr00}. Some of these techniques allow us
    to overcome the difficulties highlighted in~\cite{FreeKnotsOpenProblem96}.
    After all the necessary preliminaries, we proceed with obtaining the new
    results, formulated in section~\ref{sec:necess-cond-optim}
    (theorem~\ref{thm:main_characterization}). The proof of this result takes
    several steps. First, in section~\ref{sec:nonsmooth-optimization-approach}
    we reformulate our approximation problem as an optimisation problem and
    show that the quasidifferential of its objective function is expressed in
    terms of extreme points.  Second, in
    section~\ref{sec:confined_quasidifferential} we introduce the notion of
    \emph{confined quasidifferential}, a subset of the whole quasidifferential
    containing only the components based on the extreme points from a given
    subinterval.  Third, in section~\ref{sec:chang-vect-basis} we define an
    invertible linear transformation which simplifies the vectors from the sub-
    and superdifferentials.  Fourth, in
    section~\ref{sec:identify_stationary_subintervals} we prove
    proposition~\ref{prop:blocksubintervals}.  Our final step is to show that
    proposition~\ref{prop:blocksubintervals} and
    theorem~\ref{thm:main_characterization} are equivalent. Finally, in
    section~\ref{sec:disc-concl-remarks} we conclude and highlight future research
    directions.

\section{Preliminaries}\label{sec:preliminaries}
    \subsection{Definitions and formulations}\label{sse:definitions}

    \begin{definition}[Polynomial Spline]
      A polynomial spline is a piecewise polynomial. Each polynomial piece lies
      on an interval \([\xi_i,\xi_{i+1}],~i=0,\ldots,N-1\). The points \(\xi_0\)
      and \(\xi_N\) are the \emph{external knots}, and the points \(\xi_i\),
      (\(i=1,\dots,N-1\)) are the \emph{internal knots} of the polynomial spline.
    \end{definition}
    The spline is generally not infinitely differentiable at its knots.
    Denote the set of polynomials of degree \(m\) by \(\PP_m\) and the
    set of piecewise polynomials of degree \(m\) with \(k\) knots by
    \(\PP\PP_{m,k}\).
    \begin{definition}[multiplicity]
      An integer \(m_i\leq m+1\) is
      called the \emph{multiplicity} of the spline at the knot \(\xi_i\) if the
      spline is \(m-m_i\) times continuously differentiable at~\(\xi_i\).
    \end{definition}

    In the case examined in this paper only continuity of
    the spline is required and hence \(m_i=m,~m=1,\dots,N-1\). Therefore we consider the
    problem of finding a best approximation by a continuous piecewise
    polynomial %, which can be formulated 
    as follows:
    \begin{subequations}
      \begin{gather}
        \label{eq:optimisation_problem}
        \text{minimize } \Psi(s) \text{ subject to } s\in \PP\PP_{m,N} \cap
        C[a,b],\\
        \label{eq:main_objective_function}
        \Psi(s)=\sup_{t\in[\xi_0,\xi_N]}\max\{s(t)-f(t),-s(t)+f(t)\}.
      \end{gather}
    \end{subequations}

      \begin{definition}
        The difference between the spline and the function to approximate is called the
        \emph{deviation}.
      \end{definition}
      We denote the deviation function at point~\(t\) by
      \(\psi_t(s)=\psi(s,t)=s(t)-f(t)\).

      Our aim is to minimize the maximal absolute deviation. This maximal
      deviation occurs at points in the interval \([\xi_0,\xi_N]\) which we call
      \emph{extreme points}.

    \begin{proposition}
      \label{prop:solution_exists}
      If the function \(f\) is continuous,
      problem~\eqref{eq:optimisation_problem} admits a solution
    \end{proposition}
    \begin{proof}

    Clearly, we can restrict our search to the set
      \[
        \bigcup_ {(\xi_1,\ldots,\xi_N) \in [a,b]}\Argmin\{\Psi(s):s\in
        \PP\PP_{m}(\xi_1,\ldots,\xi_N)\}
      \]
      where \(\PP\PP_{m}(\xi_1,\ldots,\xi_N)\) is the set of polynomial splines
      with knots at \(\xi_1,\ldots,\xi_N\). This set, as a union of sets of
      solutions to the problem of best Chebyshev approximation with fixed
      knots~\cite{Nurnberger892}, is well defined. Its closure may
      contain discontinuous splines when two of the knots coincide. However, if
      any such discontinuous spline is optimal, then it is possible to construct
      a continuous spline at least as good. The proof is the same as the one
      of \cite[theorem 3.3]{Schumaker68}.
    \end{proof}

      Polynomial splines can be constructed in different ways. In this paper we
      use the truncated power function~\cite[Appendix, p. 191]{nurnberger}:
      \[{(t-\tau)}^j_+=\begin{cases}0, & \text{if } t<\tau \\ {(t-\tau)}^j, & \text{if } t\geq \tau\end{cases}.\]
      Let
      \(X=(a_{00},x_0,\xi_1,x_1,\ldots,\xi_{N-1},x_{N-1})\in\R^{(m+1)N}\), where
      \(x_i=(a_{i1},\ldots,a_{im})\in\R^m,\) \(i=0,\dots, N-1\) and
      \begin{equation}\label{eq:knot_constraints}
      a=\xi_0\leq\xi_1\leq\dots\leq\xi_{N-1}\leq\xi_N=b,
      \end{equation}
       then
      \begin{equation}\label{eq:spline_formulation_Truncated_Powers}
        s(t)=s[X](t)=a_{00}+\sum_{i=0}^{N-1} \sum_{j=1}^{m}
        a_{ij}{(t-\xi_i)}^{m+1-j}_+.
      \end{equation}
      On the \(l\)-th interval, between the knots \(\xi_{l-1}\) and \(\xi_{l}\),
      the spline \(s[X](t)\) coincides with a polynomial which we denote by~\(P_l(X,t)\):
      \begin{equation}\label{eq:polynomials}
        P_l(X,t) = a_{00}+\sum_{i=0}^{l-1} \sum_{j=1}^{m} a_{ij}{(t-\xi_i)}^{m+1-j}.
      \end{equation}

      The formulation~\eqref{eq:spline_formulation_Truncated_Powers} allows for
      the straightforward handling of constraint~(\ref{eq:knot_constraints}): it
      suffices to re-order the knots of any given spline in increasing order to
      obtain a coinciding spline satisfying this constraint. If some knots lie
      outside of the interval~\([a,b]\), they can simply be ignored (or replaced
      by knots of multiplicity~\(0\)) as they will not affect the values taken
      by the spline over this interval.

      Summarising all the above, problem~(\ref{eq:optimisation_problem}) can be reformulated as follows:
      \begin{equation}\label{eq:optimisation_problem_summary}
      \text{ minimize } \sup_{t\in [\xi_0,\xi_N]} \bigg |a_{00}+\sum_{i=0}^{N-1} \sum_{j=1}^{m}
        a_{ij}{(t-\xi_i)}^{m+1-j}_+ -f(t)\bigg |,
        \end{equation}

        \begin{equation}\label{eq:optimisation_problem_summary_constraints}
        \text{ subject to } X\in \R^{(m+1)N}.
      \end{equation}
     The problem is unconstrained and well defined, and the variables are the
     parameters of the polynomial pieces and the knots.

    \subsection{Existing work and motivation}\label{sse:existing-results}

      The theory on polynomial and fixed-knots polynomial spline approximation is
      generally
      complete~\cite{nurnberger85,rice67,sukhorukovaoptimalityfixed,sukhorukovaalgorithmfixed,tarashnin96}.
      These results are thoroughly reviewed in~\cite{nurnberger}.
      In this subsection we review the main known results for free-knot polynomial
      spline approximation.

    \subsubsection{Characterization theorems and necessary optimality conditions}

      Most characterization theorems (optimality conditions for best Chebyshev
      approximation)  are based on the notion of \emph{alternating extreme
      points} of the deviation function.
      \begin{definition}\label{def:alternating_extreme_points}
        Points \(t_1,\ldots, t_p\) are called \emph{alternating extreme points} of
        a function \(g\) defined over an interval \([a,b]\) if there exists a sign
        \(\sigma \in \{-1,1\}\) such that
        \[
        \sigma\cdot {(-1)}^i g(t_i) = \sup_{t\in [a,b]}|g(t)|.
        \]
      \end{definition}

    Traditionally, the set of polynomial splines with \(k\) free knots is defined
    as follows:
    \begin{definition}\label{def:splineNurnberger}(\cite[Appendix, definition 1.1]{nurnberger})
      Let integers \(N\geq 1\), \(m\geq 1\) and \(k\geq 1\) be given. The set of
      polynomial splines of degree \(m\) with \(k\) free knots is
      \[
        \begin{aligned}
          S_{m,k} = \bigg\{& s:[\xi_0,\xi_N] \to \R: \text{there exist points
          }\xi_0<\xi_1<\xi_2<\cdots <\xi_N \text{ and} \\ & \text{integers }
          m_1\ldots m_{N-1} \in \{1,\ldots,m+1\} \text{ with } \sum_{i=1}^{N-1}
          m_i \leq k \text{ such that} \\ & s\big|_{[\xi_i,\xi_{i+1})} \in
          \PP_m, i=0,\ldots,N-2, s\big|_{[\xi_{N-1},\xi_{N}]} \in
          \mathcal{P}_m \text{ and} \\ & s \text{ has at least \(m-m_i\)
            continuous derivatives at }\xi_i, i=1\ldots N\bigg\},
        \end{aligned}
      \]
      where \(m_i, i=1,\dots,N-1\) are the multiplicities of the corresponding
      knots.
    \end{definition}
    This definition is used when smoothness is desirable. Indeed, it
    was shown in~\cite{Schumaker68} that it is not always possible to
    approximate a function by a spline with some degree of differentiability and
    definition \ref{def:splineNurnberger} was introduced to address this problem.
    The number of knots is linked to the differentiability of the spline.
    A necessary condition for best approximation from this set is presented
      in~\cite[theorem~1.6, Appendix~1]{nurnberger}).

      \begin{theorem}\label{thm:free_knot}
        Consider a continuous function \(f\) and a polynomial spline
        \(s_0\) of degree \(m\) with knots \(\xi_0,\ldots,\xi_N\) and the corresponding multiplicities \(m_1,\dots,m_{N-1}.\) %\(k\leq\sum_{i=1}^{N-1}m_i\).
        The spline \(s_0\) is a best Chebyshev approximation in \(S_{m,k}\) to the
        function \(f\), then there exists an interval \([\xi_p,\xi_{q}]\)
        where the function \(s_0-f\) admits at least \(q-p+m+1+\sum_{i=p+1}^{q-1}m_i\)
        alternating extreme points.
      \end{theorem}
      This result has been strengthened in some cases~\cite{Mulansky92}, but
      these improvements can only be applied to the case of smooth splines.

      Definition~\ref{def:splineNurnberger} and
      theorem~\ref{thm:free_knot} allow for the number and multiplicities of
      knots to change outside of the interval~\([\xi_p,\xi_q]\). As shown in
      example~\ref{ex:too_smooth}, this may lead to some suboptimal solutions to
      satisfy theorem~\ref{thm:free_knot}.
      \begin{example}
        \label{ex:too_smooth}
        Consider the problem of approximating the piecewise linear function joining the points
        \(\{(-2,7),(-1,2),(-\frac{1}{2},-\frac{7}{8}),(0,1),(\frac{1}{2},-\frac{7}{8}),(1,2),(2,7)\}.\)
        Let \(m=3\) and \(k=1\) and consider the spline \(s(t) = |-t^3|\). This
        spline is twice differentiable at its only knot \(t=0\) and has 7
        alternating extreme points. Therefore it satisfies
        theorem~\ref{thm:free_knot}. Yet, it is not even optimal for a knot
        fixed at \(t=0\).
      \end{example}

      A major obstacle to obtain a conclusive characterization theorem has been
      the problem's nonconvexity~\cite[problem~1]{FreeKnotsOpenProblem96}. One
      tool from nonconvex analysis, the quasidifferential~\cite{dr95}, was
      successfully applied to obtain several results on fixed-knots spline
      approximation~\cite{sukhorukovaoptimalityfixed,sukhorukovaalgorithmfixed,tarashnin96}. We will extend
      these results to the case of polynomial splines with free knots.

 \subsubsection{Algorithms}
      Most existing algorithms for best Chebyshev approximation by free knot
      polynomial spline are heuristic. The following one works in two
      steps~\cite{Meinardus89}: first it finds the knots by approximating the
      function with a discontinuous spline. Then, after fixing the knots, it
      applies a Remez-like algorithm~\cite{nurnberger} to find a spline with the
      required smoothness. Although the authors demonstrate that the algorithm
      works well in practice, it is not guaranteed to converge. Indeed, the
      following example shows that it may fail to reach even a locally optimal
      solution.
      \begin{example}
        Consider the problem of approximating the following function with a
        piecewise linear spline with only one internal knot:
        \begin{equation}\label{eq:counter_example_algorithm}
          f(x) = \begin{cases}
          \sin(x) & \text{for } x \in [0,\pi]\\
          -\sin(2x) & \text{for } x \in [\pi,3\pi/2]
          \end{cases}.
        \end{equation}
        As can be seen in figure~\ref{fig:counter_example_algorithm}, the
        algorithm from~\cite{Meinardus89} does not converge towards a locally
        optimal solution. Indeed, the solution it reaches does not satisfy
        existing necessary conditions for a local best approximation~\cite{dr95}
        and~~\cite[p.20]{Dr00}.
        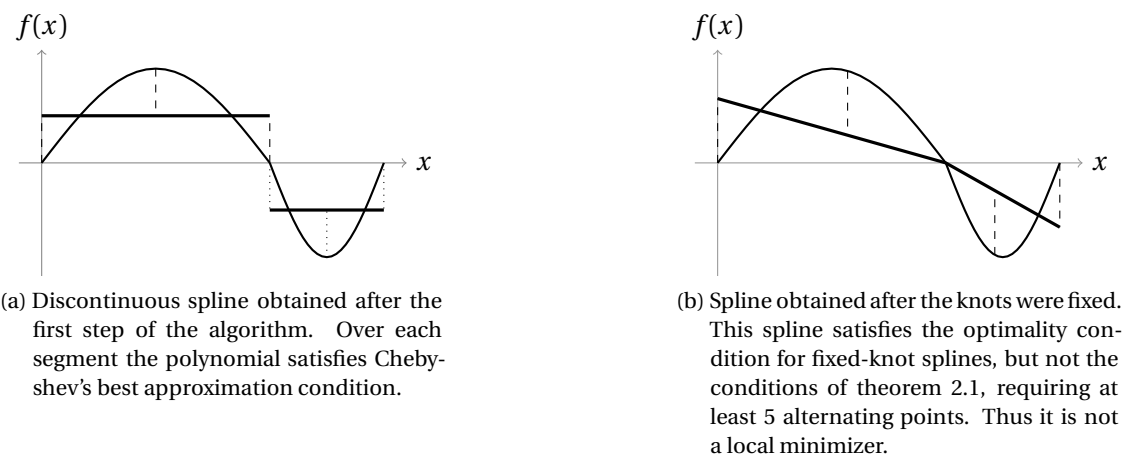
\begin{figure}[!ht]
          \subfloat[Discontinuous spline obtained after the first step of the
          algorithm. Over each segment the polynomial satisfies Chebyshev's
          best approximation condition.]{
            \centering
            \begin{tikzpicture}[xscale=1.5,yscale=1.25,domain=0:4]
              \draw[color=gray!95,->] (-0.2,0) -- (3.2,0) node[right,color=black] {\(x\)};
              \draw[color=gray!95,->] (0,-1.2) -- (0,1.2) node[above,color=black] {\(f(x)\)};

              %Draw the function to approximate:
              \draw[style=thick,domain=0:2] plot (\x,{sin(pi/2*\x r)});
              \draw[style=thick,domain=2:3] plot (\x,{-sin(pi*\x r)});
              %Draw the piecewise-linear approximation obtained with N??rnberger's
              %algorithm:
              \draw[style=very thick] (0,0.5) -- (2,0.5);
              \draw[style=very thick] (2,-0.5) -- (3,-0.5);
              %Draw the points of maximal deviation (to show that the spline is optimal if
              %knot is fixed):
              \draw[style=dashed] (0,0) -- (0,0.5) ;
              \draw[style=dashed] (1,1) -- (1,0.5) ;
              \draw[style=dashed] (2,0) -- (2,0.5) ;
              \draw[style=dotted] (2,0) -- (2,-0.5) ;
              \draw[style=dotted] (2.5,-1) -- (2.5,-0.5) ;
              \draw[style=dotted] (3,0) -- (3,-0.5) ;
            \end{tikzpicture}}
            \hfill
            \subfloat[Spline obtained after the knots were fixed. This spline
            satisfies the optimality condition for fixed-knot splines, but not
            the conditions of theorem~\ref{thm:free_knot}, requiring at
            least~5 alternating points. Thus it is not a local minimizer.]{
              \centering
              \begin{tikzpicture}[xscale=1.5,yscale=1.25,domain=0:4]
                \draw[color=gray!80,->] (-0.2,0) -- (3.2,0) node[right,color=black] {\(x\)};
                \draw[color=gray!80,->] (0,-1.2) -- (0,1.2) node[above,color=black] {\(f(x)\)};

                %Define a few variables:
                \pgfmathparse{4.4934094579090641753 - pi} \pgfmathresult \let\Xmax\pgfmathresult;% evaluate maxY
                \pgfmathparse{cos(\Xmax r)} \pgfmathresult \let\Slope\pgfmathresult;% evaluate maxY
                \pgfmathparse{sin(\Xmax r)} \pgfmathresult \let \Ymax \pgfmathresult;

                %Draw the function to approximate:
                \draw[style=thick,domain=0:2] plot (\x,{sin(pi/2*\x r)});
                \draw[style=thick,domain=2:3] plot (\x,{-sin(pi*\x r)});
                %Draw the piecewise-linear approximation obtained with N??rnberger's
                %algorithm:
                \draw[style=very thick] (0,pi*\Slope) -- (2,0) -- (3,-pi*\Slope);
                %Draw the points of maximal deviation (to show that the spline is optimal if
                %knot is fixed):
                \draw[style=dashed] (0,0) -- (0,pi*\Slope) ;
                \draw[style=dashed] (3,0) -- (3,-pi*\Slope) ;
                \draw[style=dashed] (2-2/pi*\Xmax,\Ymax) -- (2-2/pi*\Xmax,\Slope*\Xmax) ;
                \draw[style=dashed] (2+1/pi*\Xmax,-\Ymax) -- (2+1/pi*\Xmax,-\Slope*\Xmax) ;
              \end{tikzpicture}}
              \caption{An example where the algorithm fails to reach a locally optimal solution}\label{fig:counter_example_algorithm}
        \end{figure}
      \end{example}

      If a heuristic algorithm terminates at a solution which is not locally
      optimal, the results can be improved by most local optimisation methods.
      Therefore, it is crucial to have strong necessary condition for (local)
      optimality validation.  Our motivation  is to develop necessary optimality
      conditions which are stronger than the existing ones~\cite{nurnberger}.

      \subsection{Quasidifferentiable
      functions}\label{sse:quasidifferential_functions}

    \begin{definition}
      \label{def:quasidifferential}
      A function \(f\) defined on an open set \(\Omega\) is
      \emph{quasidifferentiable}~\cite{dr95,Dr00} at
      a point \(x \in \Omega\) if it is locally Lipschitz continuous, directionally
      differentiable at this point and there exists compact, convex sets
      \(\sub f(x)\) and \(\super f(x)\) such that %for any direction \(g\) the
      the derivative of \(f\) at \(x\) in any direction \(g\) can be expressed as
      \[
        f'(x,g)=\max_{\mu \in \sub{f(x)}} \langle \mu,g \rangle +
        \min_{\nu \in \super{f(x)}} \langle \nu,g \rangle.
      \]
      The sets \(\sub f(x)\) and \(\super f(x)\) are called respectively the
      \emph{sub-}
      and \emph{superdifferential} of the function \(f\) at the point \(x\). The pair \([\sub
        f(x),\super f(x)]\) is called a \emph{quasidifferential} of the function \(f\) at
      the point \(x\).
    \end{definition}

      At any local minimizer \(x^* \in \Omega\)
      of a quasidifferentiable function \(f\) we have~\cite{dr95,Dr00}
      \begin{equation}\label{eq:quasidiffcondition}
        -\overline{\partial}f(x^*)\subset \underline{\partial}f(x^*).
      \end{equation}
      A point \(x^*\) satisfying condition~\eqref{eq:quasidiffcondition} is an
      \emph{inf-stationary} point.

      The only points where the spline function can be nonsmooth are its knots.
      The discontinuity of the derivative at the knot \(\xi_l\) is determined by
      the value of \(a_{l\,m}\): if \(a_{l\,m}=0\), then the spline is differentiable
      at the knot \(\xi_l\). We say that the knot is a \emph{neutral knot}. If
      \(a_{l\,m}>0\), then around the knot \(\xi_{l}\) the spline
      behaves like the maximum of two linear functions. We call such a knot a
      \emph{max-knot}. Finally, if \(a_{l\,m}<0\), then in the neighbourhood of knot
      \(\xi_{l}\) the spline behaves like the minimum of two linear functions and
      the knot is called a \emph{min-knot}.

      We use the following notations.
      \begin{itemize}
        \item \(E_\mathrm{smooth}\) is the set of extreme points of the deviation
              function outside of internal knots of the spline;
              \item \(E_\mathrm{neutral}\) is the set of neutral knots;
        \item \(E_\mathrm{max}\) is the set of max-knots; and
        \item \(E_\mathrm{min}\) is the set of min-knots.
      \end{itemize}

      In the following section we will formulate the main result of this paper.

  \section{Characterization through quasidifferentiability}\label{sec:necess-cond-optim} 
  
  The necessary condition relies on the following definition: 
  \begin{definition}\label{def:stable_unstable_points} 
    Min-knots with positive maximal deviation and max-knots with negative
    maximal deviation are called \emph{unstable knots}. Max-knots with positive
    maximal deviation and min-knots with negative maximal deviation are called
    \emph{stable knots}.
    \end{definition}

    The following theorem is our main result, which is an improvement to the existing ones in the case when only continuity is desired. 
    \begin{theorem}\label{thm:main_characterization}
      A spline satisfies condition~\eqref{eq:quasidiffcondition} over the
      interval \([\xi_0,\xi_N]\) if and only if there exists a subinterval
      \([\xi_p,\xi_{q}]\) containing a sequence of \(m(q-p)+2+l\) alternating extreme points of
      the deviation function, where~\(l\) is the number of non-neutral internal
      knots inside \((\xi_p,\xi_q)\). The end-points \(\xi_p\) and \(\xi_{q}\) may be
      included in this sequence only if they are not unstable.
    \end{theorem}
    \begin{proof}
      The proof of theorem~\ref{thm:main_characterization} proceeds as follows:
      \begin{itemize}
        \item First,  in section~\ref{sec:nonsmooth-optimization-approach}
              we show that the quasidifferential of
              function~\eqref{eq:main_objective_function} is expressed in
              terms of extreme points.
        \item Second, in section~\ref{sec:confined_quasidifferential}
              we introduce the notion of \emph{confined quasidifferential}, a
              subset of the whole quasidifferential containing only the components
              based on the extreme points from a given subinterval.
        \item Third, in section~\ref{sec:chang-vect-basis} we define an invertible linear transformation which simplifies the
        vectors from the sub- and superdifferentials.
        \item Fourth, in section~\ref{sec:identify_stationary_subintervals} we
          prove proposition~\ref{prop:blocksubintervals}, which is equivalent
          to~(\ref{eq:quasidiffcondition}) and therefore it gives another
          necessary and sufficient stationarity condition.  
        \item Our final step is to show that
          proposition~\ref{prop:blocksubintervals} and
          theorem~\ref{thm:main_characterization} (our main result) are
          equivalent and therefore our main result is equivalent to the
          stationarity condition in the sense of Demyanov-Rubinov.\qedhere
      \end{itemize}
    \end{proof}

  \section{Quasidifferential of the objective function}\label{sec:nonsmooth-optimization-approach}

    \subsection{Quasidifferential of the spline functions}\label{sse:quasidifferential-spline-function}

     The aim of this subsection is to analyse the
      function \(\Psi\) as a function of \(X\). % We define the function
      %\(\psi_t(X)=\psi(X,t)\).
      As pointed above, when the point \(t\) is a knot,
      the function \(\psi_t(s[X])\) may not be differentiable.

      \subsubsection{Case when the point \(t\) is not a knot}\label{sec:Quasidif_Spline_Not_Knot}

        If there exists an index \(l\) such that \(\xi_l<t<\xi_{l+1}\), then the
        function \(\psi_t\) is differentiable:
        \begin{equation}\label{eq:polynomial_gradient}
          \nabla\psi_t(s[X]) = \nabla_X P_l(X,t),
        \end{equation}
        where \(\nabla_X P_l(X,t)=\partial P_l(X,t)/\partial X\).

      \subsubsection{Case when point \(t\) is a knot}\label{sec:Quasidif_Spline_Knot}

        The function \(\psi_t(s[X])\) can be nondifferentiable at \(t=\xi_l\) for
        some \(l\in \{1,\ldots,N-1\}\). In order to compute its quasidifferential we
        isolate the nonsmooth part by rewriting the function as follows:
        \(\psi_t(s[X])=\zeta_l^t(X)+\gamma_l^{t}(X)\), where
        \begin{align}
          \zeta_{l}^{t}(X)&= a_{lm} {(t-\xi_l)}_+, \nonumber \\
          \gamma_l^t(X)&=
          a_{00} +  \sum_{i=0}^{l-1}\sum_{j=1}^{m}
          a_{ij}{(t-\xi_i)}^{m+1-j}_+  & [&= \mathrm{P_l(X,t)}\text{ since }\mathrm{\xi_{l-1}<t}] \nonumber\\
          &\quad +\sum_{j=1}^{m-1} a_{l j}
          {(t-\xi_l)}^{m+1-j}_+ -f(t)\nonumber\\ & \quad+
          \sum_{i=l+1}^N\sum_{j=1}^{m}
          a_{ij}{(t-\xi_i)}^{m+1-j}_+ &[&=
          \mathrm{0}\text{ since } \mathrm{t<\xi_{l+1}}]
          \nonumber\\
          &= P_l(X,t) + \sum_{j=1}^{m-1} a_{l j}{(t-\xi_l)}^{m+1-j}_+ -f(t). \label{eq:gamma_function}
        \end{align}

        \paragraph{Quasidifferential of function \(\gamma_l^t\):}

        At \(t=\xi_l \), the function \(\gamma_l^t\) is differentiable with respect
        to \(X\), and its gradient coincides with the gradient of \(P_l(X,t)\):
        \begin{equation}
          \nabla\gamma_l^t(X) = \nabla_X P_l(X,t).\label{eq:gamma_gradient}
        \end{equation}

        \paragraph{Quasidifferential of function \(\zeta_l^t\):}
        The function \(\zeta_l^t\) may not be differentiable. As a product of a
        constant and a maximum of two linear functions its quasidifferential depends
        on the sign of \(a_{l m}\):
          \begin{subequations}\label{eq:zeta_quasidifferential}
            \begin{align}
              \intertext{if \(a_{lm}<0\), then}
                \label{eq:zeta_min}
                  &\begin{aligned}
                    \underline{\partial}\zeta_l^t(X) &= \{\zero_{N(m+1)}\}\\
                    \overline{\partial}\zeta_l^t(X) &=
                    \co\{\zero_{N(m+1)},{\left(\zero_{l(m+1)-1},-a_{l
                          m},\zero_{(N-l)(m+1)}\right)}^T\},
                  \end{aligned}
              \intertext{if \(a_{lm}=0\) then}
                      &\begin{aligned}
                        \underline{\partial}\zeta_l^t(X) &= \{\zero_{N(m+1)}\}\\
                        \overline{\partial}\zeta_l^t(X) &= \{\zero_{N(m+1)}\}.
                      \end{aligned}\label{eq:zeta_smooth}
                      \intertext{Notice that in this case the function \(\psi_t\) is
                    differentiable at \(t=\xi_l\).}
              \intertext{if \(a_{lm}>0\), then}
                    \label{eq:zeta_max}
                      &\begin{aligned}
                        \underline{\partial}\zeta_l^t(X) &=
                        \co\{\zero_{N(m+1)},{\left(\zero_{l(m+1)-1},-a_{l
                        m},\zero_{(N-l)(m+1)}\right)}^T\},\\
                        \overline{\partial}\zeta_l^t(X) &= \{\zero_{N(m+1)}\}.
                      \end{aligned}
                    \end{align}
          \end{subequations}
        Here \(\co \) represents the convex hull of the corresponding set, \(\zero_n\) is an \(n-\)dimensional column vector.

        \paragraph{Combine  \(\gamma_l^t\) and \(\zeta_l^t\):}
        Finally, notice that at \(t=\xi_l\)
        \begin{align*}
          \nabla_X P_{l+1}(X,t)-\nabla_X P_l(X,t) &= \nabla_X
          (P_{l+1}-P_{l})(X,t)\\&={(\zero_{l(m+1)-1},-a_{lm},\zero_{(N-l)(m+1)})}^T.
        \end{align*}
        % \[
        % \nabla_X P_{i+1}(X,t)-\nabla_X P_i(X,t) = \left(0,\zero_{(i-1)(m+1)},\eta_i(t),\zero_{(N-i-1)(m+1)}\right)^T,
        % \]
        % and that at the point \(\xi_i\) we have \(\eta_i(\xi_i)=(\zero_m,-a_{i1})\).
        Therefore, combining equation~\eqref{eq:gamma_gradient} with~\eqref{eq:zeta_quasidifferential} using
        quasidifferential calculus~\cite[Chapter 1, p.11]{Dr00}, we obtain
          \begin{subequations}\label{eq:psi_quasidifferential}
            \begin{align}
              \intertext{if \(a_{l m}<0\), then}
                      &\begin{aligned}
                        \underline{\partial}\psi_t(s[X]) &= \{\zero_{N(m+1)}\},\\
                        \overline{\partial}\psi_t(s[X]) &=
                        \co\{\nabla_X P_{l}(X,t),\nabla_X P_{l+1}(X,t)\};
                      \end{aligned}\label{eq:quasidifferential_psi_min}
              \intertext{if \(a_{lm}=0\) then}
                      &\begin{aligned}
                        \underline{\partial}\psi_t(s[X]) &=
                        \{\nabla_X P_{l}(X,t)\} = \{\nabla_X P_{l+1}(X,t)\},\\
                        \overline{\partial}\psi_t(s[X]) &= \{\zero_{N(m+1)}\};
                      \end{aligned}\label{eq:quasidifferential_psi_smooth}
              \intertext{if \(a_{lm}>0\), then}
                      &\begin{aligned}
                        \underline{\partial}\psi_t(s[X]) &= \co\{\nabla_X P_{l}(X,t),\nabla_X P_{l+1}(X,t)\},\\
                        \overline{\partial}\psi_t(s[X]) &= \{\zero_{N(m+1)}\}.
                      \end{aligned}\label{eq:quasidifferential_psi_max}
            \end{align}
          \end{subequations}

\subsection{Continual maximum}\label{sse:continual max}
    The function \(\Psi(s)\) defined in~\eqref{eq:main_objective_function} is a continual
    maximum of functions. The quasidifferential properties and calculus of such
    functions have been studied in~\cite{luderer86}. According
    to~\cite[theorem~1]{luderer86}, the function \(\Psi(s)\) is
    quasidifferentiable at a point \(s_0\) if:
    \begin{enumerate}
      \item the function under the supremum, \(|\psi(s,t)|\), is
            continuous in \(t\) for any \(s\) from the neighbourhood of \(s_0\);
      \item the function \(|\psi(s,t)|\) is uniformly directionally differentiable
            at the point~\(s_0\) for any \(t\in [\xi_0,\xi_N]\);
      \item the function \(|\psi(s,t)|\) is quasidifferentiable with respect to
            \(s\) at the point \(s_0\) and for any extreme point \(t\) there
            exists a pair of
            convex compact sets \(B(s_0)\) and \(A_t(s_0)\) such that \(B(s_0)=\super
            |\psi_t(s_0)|+A_t(s_0)\).
    \end{enumerate}

    The first condition is verified whenever the function \(f(t)\) is
    continuous. The second condition is also verified, because the function
    \(\psi(s,t)\) is locally Lipschitz continuous.
    % Hence, the function
    %\(\psi_t(s[X])\) is quasidifferentiable if the spline \(s[X]\) is
    %quasidifferentiable.
To verify the third condition, we need to calculate the
    quasidifferential of the function \(\psi_t(s[X])\), and therefore of the
    spline~\(s[X]\). We will verify this condition in subsection~\ref{sse:quasidifferential-objective-function}.

    %These properties depend on the functions under the supremum. Let \(\psi(s,t)
    %= s(t)-f(t)\). We assume that function \(f(t)\) is continuous, which then
    %implies that the function \(\psi(s,t)\) is continuous in \(t\) for any \(s\).
    %Since the function \(\psi_t(s)=\psi(s,t)\) is uniformly directionally
    %differentiable and quasidifferentiable at a point \(s\) for any \(t\in
    %[a,b]\), by~\cite[theorem~1]{luderer86}, the function \(\Psi(t)\) is
    %quasidifferentiable. The quasidifferential of this function depend on the
    %quasidifferential of the function \(\psi_t(s)\) (and of \(-\psi_t(s)\)) at
    %the points \(t\) where the supremum is attained.

    \subsection{Quasidifferential of the objective function}\label{sse:quasidifferential-objective-function}

      \subsubsection{Explicit formulation of the quasidifferential of the objective function}\label{sss:quasidifferential}

        We denote the index of the interval containing the extreme point \(t\) by
        \(j_t\). If \(t\) is a knot, then it joins the \(j_t\)-th and \((j_t+1)\)-st
        intervals.

        We apply quasidifferential calculus~\cite[Chapter 1, formula (5.3)]{Dr00}
        to obtain the quasidifferential of function~\(|\psi_t(s)|\) from that of
        function~\(\psi_t(s)\). Since there are a finite number of points (all of them are
        knots) where the superdifferential is not zero, it is easy to construct the
        sets \(B(s)\) and \(A_t(s)\) required to fulfil the third requirement
        of~\cite[theorem 1]{luderer86}. One way to construct these sets is as follows:
        \begin{itemize}
        \item \(t\) is not an internal knot, then $$B(s_0)=A_t(s_0)=\sum_{i=1}^{N}\super
            |\psi_{\xi_i}(s_0)|;$$
        \item \(t=\xi_j, 1<j<N\) is an internal knot, then
        $$B(s_0)=\sum_{i=1}^{N}\super
            |\psi_{\xi_i}(s_0)|,\quad  A_t(s_0)=\sum_{i=1}^{j-1}\super
            |\psi_{\xi_i}(s_0)|+\sum_{i=j+1}^{N}\super
            |\psi_{\xi_i}(s_0)|.$$
        \end{itemize}
        Let us denote the sets of indices of the extreme points of the deviation
        function respectively with positive and negative deviation by \(K^+=\{t:
        \psi(s,t)=\Psi(s)\}\) and \(K^-=\{t: \psi(s,t)=-\Psi(s)\}\). Function
        \(\Psi\) admits the following superdifferential:
        \begin{subequations}\label{eq:superdifferential_Psi}
          \begin{equation}\label{eq:super_sum}
            \overline{\partial}\Psi(s[X])=\{0_{2N}\}+\co(\Sigma^+-\Sigma^-),
          \end{equation}
          where
          \begin{align}
            \Sigma^+&=  \sum_{t\in K^+\cap E_\mathrm{min}}\co\{\nabla_X P_{j_t}(X,t),\nabla_X P_{j_t+1}(X,t)\}\label{eq:super:K+Lm}\\
            \Sigma^-&= \sum_{t\in K^-\cap E_\mathrm{max}}\co\{\nabla_X P_{j_t}(X,t),\nabla_X P_{j_t+1}(X,t)\}.\label{eq:super:K-LM}
          \end{align}
        \end{subequations}
        and the corresponding subdifferential:
        \begin{subequations}\label{eq:subdifferential_Psi}
          \begin{align}
            \underline{\partial}\Psi(s[X])=\co&\Bigg\{
            \bigcup_{t\in K^+\cap (E_\mathrm{smooth}\cup E_\mathrm{neutral})}(\nabla_X P_{j_t}(X,t)-\overline{\partial}\Psi(s[X])), \label{eq:sub:K+Li}\\
            &\bigcup_{t\in K^-\cap (E_\mathrm{smooth}\cup E_\mathrm{neutral})}(-\nabla_X P_{j_t}(X,t)-\overline{\partial}\Psi(s[X])),\label{eq:sub:K-Li}\\
            &\bigcup_{t\in K^+\cap E_\mathrm{max}}(\co\{\nabla_X P_{j_t}(X,t),\nabla_X P_{j_t+1}(X,t)\}-\overline{\partial}\Psi(s[X])),\label{eq:sub:K+LM}\\
            &\bigcup_{t\in K^-\cap E_\mathrm{min}}(-\co\{\nabla_X P_{j_t}(X,t),\nabla_X P_{j_t+1}(X,t)\}-\overline{\partial}\Psi(s[X]))\label{eq:sub:K-Lm},\\
            &\bigcup_{t\in K^+\cap E_\mathrm{min}}(0_{2n}-\sum_{\mathclap{
              \substack{\tau\in K^+\cap E_\mathrm{min}\\\tau\neq t}}}\co\{\nabla_X P_{j_t}(X,\tau),\nabla_X P_{j_t+1}(X,\tau)\}+\Sigma^-),\label{eq:sub:K+Lm}\\
              &\bigcup_{t\in K^-\cap E_\mathrm{max}}(0_{2n}+\sum_{\mathclap{
                \substack{\tau\in K^-\cap E_\mathrm{max}\\\tau\neq t}}}\co\{\nabla_X P_{j_t}(X,\tau),\nabla_X P_{j_t+1}(X,\tau)\}-\Sigma^+)\label{eq:sub:K-LM}%,\\
              \Bigg\}.
          \end{align}
        \end{subequations}

        \begin{remark}
          Only unstable extreme points contribute to the superdifferential.
        \end{remark}

        Finally we introduce the following notation. Let
        \begin{equation}\label{beta}
        \beta_{j_t}(X,t) = \sign(\Psi(s[X],t))\cdot \nabla_X P_{j_t}(X,t).
        \end{equation}
        To the smooth extreme points (non internal knots), neutral knots and stable knots we associate the set
        \[
        \S=\{\beta_{j_t}(X,t): t\in E_\mathrm{smooth}\cup E_\mathrm{neutral}\cup(K^+\cap E_\mathrm{max}) \cup (K^-\cap E_\mathrm{min}) \}.\
        \]
        To the extreme points coinciding with unstable knots we associate the following sets:
        \begin{align*}
          \Delta_{j_t}&=\co\{\beta_{j_t}(X,t), \beta_{j_t+1}(X,t)\};\\
          C_\Delta&=-\sum_{\Delta_{j_t}\neq \Delta} \Delta_{j_t}.
        \end{align*}
        Define \(\U=\{\Delta_{j_t}: t\in (K^+\cap E_\mathrm{min}) \cup
        (K^-\cap E_\mathrm{max})\}\).
        %We denote by \(\U=\{\Delta_k: \xi_k = \tau, \tau\in (K^+\cap E_\mathrm{min}) \cup
        %(K^-\cap E_\mathrm{max})\}\) the set of these sets.
        %\end{itemize}
        The quasidifferential of the function \(\Psi\) is
        \(\overline{\partial}\Psi = \sum_{\Delta_{j_t}\in \U} \Delta_{j_t}\) and
        %\(\underline{\partial}\Psi=\co\{\bigcup_{d\in \S}d-\overline{\partial}\Psi,\cup_{\Delta\in \U}(\zero+ C_\Delta)\}\).
        \(\underline{\partial}\Psi=\co\{\S-\overline{\partial}\Psi,\cup_{\Delta\in \U}(\zero+ C_\Delta)\}\).

  \section{Confined quasidifferential}\label{sec:confined_quasidifferential}

    In this section we show that the
    inclusion~\eqref{eq:quasidiffcondition} is verified when it is verified on a subinterval.
    We start by introducing the confined quasidifferential.
    Let \(p\) and \(q\) be given indices. Consider the following notation.
    We define \(\S_p^q=\{\beta_{j_t}(X,T)\in \S: t\in
      [\xi_p,\xi_q]\}\) and \(\U_p^q=\{\Delta_v\in \U: \xi_v\in
      (\xi_p,\xi_q)\}\) the respective subsets of elements of \(\S\) and \(\U\)
    corresponding to the extreme points lying in the interval~\([\xi_p,\xi_q]\).
    Similarly, ${C_{p}^{q}}_\Delta=-\sum_{\Delta_{j_t}\in
      \U_p^q, \Delta_{j_t}\neq \Delta} \Delta_{j_t}$.

    \begin{definition}\label{def:confined_subdifferential}
      The quasidifferential
      \([\underline{\partial}\Psi,\overline{\partial}\Psi\)] confined to the
      interval \([\xi_p,\xi_q]\) is defined by
      \begin{subequations}\label{eq:confined_quasidifferential}
        \begin{align}\label{eq:confined_superdifferential}
          \overline{\partial_p^q} \Psi&=\sum_{\Delta \in \U_p^q} \Delta,\\\label{eq:confined_subdifferential}
          %\underline{\partial_p^q} \Psi&=\co\bigg\{\bigcup_{d\in \S\mathrlap{_p^q}}(d-\overline{\partial_p^q} \Psi),
          \underline{\partial_p^q} \Psi&=\co\bigg\{\S_p^q-\overline{\partial_p^q} \Psi,
          \bigcup_{\Delta\in \U\mathrlap{_p^q}}{C_p^q}_\Delta\bigg\}.
        \end{align}
      \end{subequations}
    \end{definition}

    \subsection{Stationary subintervals}\label{sse:stationary-subinterval}

    \begin{definition}
      \label{def:sationary_interval}
      An interval \([\xi_p,\xi_q]\) is \emph{stationary} if
        \begin{equation}\label{eq:confined_inclusion}
          -\overline{\partial_p^q} \Psi(s)\subset \underline{\partial_p^q} \Psi(s)
        \end{equation}
    \end{definition}

    Denote by \(\coprod_{i=1}^q A_i\) the collection of sets composed of one element per set \(A_i\):
    \[
      \{\delta_1 \ldots \delta_q\} \in \coprod_{i=1}^q A_i \Leftrightarrow (\delta_1,\dots,\delta_q) \in \prod_{i=1}^q A_i
    \]
    and \(\prod\) represents the Cartesian product.
  \begin{proposition}\label{prop:blocksubintervals}
    The interval \([\xi_p,\xi_q]\) is stationary if and only if
    \begin{equation}
      \label{eq:zero_in_hull_gradients}
      \zero \in \co \big\{\S_p^q,\C\big\},
        \forall \C\in\coprod_{\Delta\in \U_p^q} \Delta,
    \end{equation}
  \end{proposition}

    Before proceeding to prove this proposition we recall the following result:
    \begin{lemma}
      \label{lem:sum_convex_sets_inclusion}
      Let \(A\) and \(B\) be convex sets and \(C\) a compact set. If
      \(B+C \subset A+C\), then \(B\subset A\).
    \end{lemma}
    \begin{proof}
      It suffices to show that no hyperplane can separate the set \(A\) from
      any subset of \(B\).
      Consider a nonzero vector \(u\in \R^n, u\neq 0\) and a scalar \(\alpha\in \R\) such that \(\langle u,a\rangle > \alpha\),
        for all \(a\in A\) and take any point \(b\in B\).
      Define \(c_0\in C\) such that \(\langle u,c_0\rangle = \min\{\langle
      u,c\rangle: c\in C\}\) and let \(d=b+c_0\). Since \(d\in B+C\subset A+C\), we can find \((a,c)\in A\times
      C\) such that \(d=a+c\), and so
      \[
        \langle u,b\rangle = \langle u,d-c_0\rangle = \langle u, a\rangle + \langle u,c-c_0\rangle > \alpha.\qedhere
      \]
    \end{proof}

 \begin{proof}[Proof of Proposition~\ref{prop:blocksubintervals}]
    First assume that the interval \([\xi_p,\xi_q]\) is stationary and
     take \[\C = \bigcup_{\Delta\in \U_p^q}\{\delta_\Delta\}\in \coprod_{\Delta\in \U_p^q} \Delta,~\mathrm{where}~\delta_\Delta~\mathrm{is~a~convex~combination~of}~\beta_{j_t}~\mathrm{and}~\beta_{j_t+1}.\]
     Let \(b\in -\overline{\partial_p^q} \Psi \subset
       \underline{\partial_p^q} \Psi\). There exist \(d_\Delta\in
     {C_p^q}_\Delta\) for each \(\Delta\in \U_p^q\), \(b_d\in
     \overline{\partial_p^q}\Psi\) for each \(d\in
     \S_p^q\), and associated \(\alpha_d\geq 0, \alpha_\Delta\geq 0, \sum_{d\in \S_p^q} \alpha_d
     + \sum_{\Delta\in \U_p^q}\alpha_\Delta=1\) such that
   \begin{align*}
       b&=\sum_{d\in \S_p^q} \alpha_d (d-b_d) + \sum_{\Delta\in \U_p^q}
       \alpha_\Delta d_\Delta.\\
       &=\sum_{d\in \S_p^q} \alpha_d d + \sum_{\Delta\in \U_p^q} \alpha_\Delta \delta_\Delta -
       \sum_{\Delta\in \U_p^q} \alpha_\Delta (\delta_\Delta - d_\Delta) - \sum_{d\in \S_p^q}
       \alpha_d b_d.
     \end{align*}
     Since \(\overline{\partial_p^q}\Psi = \Delta - {C_p^q}_\Delta\) for any
     \(\Delta\in \U_p^q\), and by the convexity of the confined subdifferential,
     \[
       \sum_{\Delta\in \U_p^q} \alpha_\Delta (\delta_\Delta - d_\Delta) + \sum_{d\in
         \S_p^q}\alpha_d b_d \in \overline{\partial_p^q}\Psi.
     \]
     Hence, since \(b\) was arbitrary,
     \(
       -\underline{\partial_p^q} \Psi  \subset \co
       \big\{\S_p^q,\C\big\} -\overline{\partial_p^q}\Psi.
     \)
     Applying lemma~\ref{lem:sum_convex_sets_inclusion} we conclude that
     \( \zero \in \co \big\{\S_p^q,\C\big\}\).

    Now assume formula~\eqref{eq:zero_in_hull_gradients} and let
     \(b\in \overline{\partial_p^q} \Psi\). There exists \(\delta_\Delta\in
       \Delta, \Delta\in \U_p^q\) such that \(b= \sum_{\Delta\in \U_p^q}
       \delta_\Delta\). Then, by assumption there exist \(d\in \co~\S_p^q\) and associated $\alpha_d, \alpha_\Delta (\Delta\in\U_p^q),$ such that $\alpha_d+\sum_{\Delta\in\U_p^q}\alpha_{\Delta}=1,$
     \begin{align*}
       -b &= \alpha_d d + \sum_{\Delta\in \U_p^q} \alpha_{\Delta} \delta_\Delta -
       b\\
       & = \alpha_d(d-b) + \sum_{\Delta\in \U_p^q}
       \alpha_{\Delta}(\delta_\Delta-b)\in \underline{\partial_p^q} \Psi.
     \end{align*}
     Therefore the interval \([\xi_p,\xi_q]\) is stationary.
  \end{proof}

      \begin{corollary}\label{cor:confined_inclusion_equivalent_to_inf_stationarity}
        A spline \(s[X]\) is an inf-stationary solution to the
        problem~\eqref{eq:optimisation_problem} if and only if there exists
        a stationary subinterval.
      \end{corollary}
      \begin{proof}
        By definition, a spline \(s[X]\) is an inf-stationary solution if and
        only if the interval \([\xi_0,\xi_m]\) is stationary. Since for any
        \(\C = \bigcup_{\Delta\in \U}\{\delta_\Delta\}\in \coprod_{\Delta\in \U} \Delta\)
        and any \(0\leq p \leq q\leq m\)
        \[
      \co \Bigg\{\S_p^q,\bigcup_{\Delta\in \U_p^q} \delta_\Delta\Bigg\}
      \subset \co \big\{\S,\C\big\},
        \]
        the proof immediately follows proposition~\ref{prop:blocksubintervals}.
      \end{proof}

  \section{Auxiliary linear transformation and necessary inclusion}\label{sec:chang-vect-basis}

    The quasidifferential of the function~\eqref{eq:main_objective_function} is based on the gradients of
    the polynomials \(P_i(X,t)\):
    \[\nabla_X P_i(X,t) =
    \begin{pmatrix}
      1\\
      \eta_1(t)\\
      \mu_2(t)\\
      \eta_2(t)\\
      \vdots \\
      \mu_{i}(t) \\
      \eta_{i}(t)\\
      \zero
    \end{pmatrix}
    ,\]
    where
    \begin{align}\label{eq:eta_l}
      \eta_l(t) &= (t-\xi_{l-1},\ldots,{(t-\xi_{l-1})}^m)^T,
      l=1,\dots,N,\\
      \mu_1(t)&= 1\label{eq:mu_1}\\
      \mu_l(t) &=- \sum_{j=1}^m
      (m-j+1)a_{l-1\;j}{(t-\xi_{l-1})}^{m-j}\nonumber\\
      &=-a_{l-1\;m}-\sum_{j=1}^{m-1}(m-j+1)a_{l-1\;j}{(t-\xi_{l-1})}^{m-j},
      l=2,\dots,N\nonumber\\
      &=-a_{l-1\;m}-\langle v_l^T , \eta_l(t)\rangle~ l=2,\dots,N,\label{eq:mu_l}
      \end{align}
    where \(v_l=(2a_{l-1\;m-1},3a_{l-1\;m-2}, \dots, ma_{l-1\;1},0).\)

    Since the coefficients of \(\eta_l(t)\) are powers of \((t-\xi_i)\) and
    \(\mu_l(t)\) is a linear combinations of these powers, it is possible to use
    binomial expansion to define a linear transform which sets most coefficients to
    \(\zero\) and only leaves nonzero the elements corresponding to the
    block interval to which \(t\) belongs. We provide details below.

    Define the matrices:
    \[
    V_1=I_{m+1}, \quad
    V_l =
    \begin{pmatrix}
      -1& -v_l  \\
      0 & I_m
    \end{pmatrix} \in \R^{(m+1)\times (m+1)}, l=2,\dots,N,
    \]
    \[
    V =
    \begin{pmatrix}
      V_1     & & \makebox(0,0){\text{\huge0}}      \\
       & \ddots & \\
      \text{\huge0}       & & V_n    \\
    \end{pmatrix} \in \R^{n(m+1)\times n(m+1)}.
    \]
    Then we find that
    \[
    V_l
    \begin{pmatrix}
      \mu_l(t) \\ \eta_l(t)
    \end{pmatrix} =
    \begin{pmatrix}
      a_{l\;m} \\ \eta_l(t)
    \end{pmatrix}
    \]
    and
    \[
    V\nabla_X P_i(X,t) =
    \begin{pmatrix}
      1\\
      \eta_1(t)\\
      a_{2\;m}\\
      \eta_2(t)\\
      \vdots \\
      a_{i\;m} \\
      \eta_{i}(t)\\
      \zero
    \end{pmatrix}.
    \]

    Consider an arbitrary point~\(t\). Let us first notice the following
    equality
    \begin{align*}
      {(t-\xi_{p})}^j =& {(t-\xi_q+\xi_q-\xi_p)}^j\\
      %&= \sum_{k=0}^j C_j^k{(\xi_l-\xi_{l-1})}^{j-k}{(t - \xi_{l})}^k.\\
      = &\sum_{k=0}^j \binom{k}{j}{(\xi_q-\xi_{p})}^{j-k}{(t - \xi_{q})}^k, j=1,\dots,m.\\
     % &= \langle w_{pq}^j, \begin{pmatrix}
     % a_{q\;m} \\ \eta_q(t)
    %\end{pmatrix} \rangle ,
      %& = {(\xi_l-\xi_{l-1})}^j +
      %{(C_j^1{(\xi_l-\xi_{l-1})}^{j-1},\ldots,C_j^j{(\xi_l-\xi_{l-1})}^{j-j})}^T\eta_{l+1}.
    \end{align*}
   Consider the following vectors: % where
    \begin{align*}
   w_{pq}^0 &= \bigg(\frac{a_{pm}}{a_{qm}},0,\cdots,0\bigg),\\
    w_{pq}^j &= \bigg(\frac{1}{a_{qm}}(\xi_q-\xi_p)^j ,
    \binom{1}{j}(\xi_q-\xi_p)^{j-1} , \cdots ,  \binom{j-1}{j}(\xi_q-\xi_p)  , 1
    , 0 ,\cdots ,0\bigg).
     \end{align*}

For each index~\(p\) define the corresponding index~\(q\) in the following way:
\begin{equation}
      \label{eq:closest_nonzero_alm}
      q=k(p) = \min \{l> p: a_{l\;m}\neq 0\}.
    \end{equation}
    That is, for each~\(\xi_p\) we find the index~\(q\) of the next
    nonsmooth (non-neutral) knot.

    Define the matrices \(W_{pr}\in \R^{(m+1)\times (m+1)}\, p=1\dots N, 
      r=p+1\dots N \) as follows:
    \begin{itemize}
    \item   the rows of \(W_{pr}\) are the row vectors~\(w^j_{pr}\)
      (\(j=0\ldots m\)) if~\(r=q=k(p)\)
    \item   \(W_{pr}=\zero\) otherwise.
    \end{itemize}

 If \(p\) and \(q\) satisfy
    equation~\eqref{eq:closest_nonzero_alm}, then
    \[
    \begin{pmatrix}
      a_{p\;m} \\ \eta_p(t)
    \end{pmatrix}= W_{pq}
    \begin{pmatrix}
      a_{q\;m} \\ \eta_q(t)
    \end{pmatrix}.
    \]
    Define also
    \[
    W =
    \begin{pmatrix}
     I_{m+1} & -W_{12} & \cdots  & -W_{1m}  \\
     & \ddots & \ddots & \\
        \text{\huge0}  & & I_{m+1} & -W_{(m-1)m}\\
          & & & I_{m+1}\\
    \end{pmatrix} \in \R^{nm\times nm}.
    \]

    Note that for each index~\(p\) there exists only one index~\(q=k(p),\) such that the corresponding~\(W_{p q}\) is nonzero.

    %\[
    %W =
    %\begin{pmatrix}
    % I_m & -W_{12} & \cdots & \cdots & -W_{1m}  \\_
    %         & I_m   &-W_{23} & \cdots &  -W_{2m} \\
    %         &  & \ddots & \ddots & \vdots \\
    %    & \text{\huge\zero}  & & I_m & -W_{(m-1)m}\\
    %    &   & & & I_m\\
    %\end{pmatrix} \in \R^{nm\times nm}.
    %\]

    Both \(V\) and \(W\) are triangular matrices with nonzero elements on the
    diagonal, and so they are both full rank. Furthermore, for any \(t\) we
    have:
    \[
    WV\nabla_X P_l(X,t) =
    \begin{pmatrix}
      \zero\\
      a_{p\;m}\\
      \eta_p(t)\\
      \vdots \\
      a_{q\;m} \\
      \eta_{q}(t)\\
      \zero
    \end{pmatrix},
    \]
    where \(q\) is defined as in \eqref{eq:closest_nonzero_alm} and therefore all the knots between~\(\xi_p\) and~\(\xi_q\) are neutral.

Let us introduce the following definitions.
    \begin{definition}
  Subintervals \([\xi_{i-1},\xi_i], i=1,\dots,n\) are called \emph{unit
    subintervals}. Intervals delimited by non-neutral external knots and whose
  internal knots are neutral are called \emph{block subintervals}.
  \end{definition}

Consider now the linear transformation~\(M=WV\) and define the following sets for
any block subinterval \([\xi_p,\xi_q]\)
    \begin{align}
      \A_p^q~\mathrm{is~the~block~of~nonzero~coordinates~of}~ \bigcup_{d\in \S_p^q} (Md)\\%|[\xi_p,\xi_q]}\\
      {\B_p^q}_\Delta\mathrm{is~the~block~of~nonzero~coordinates~of}~ \bigcup_{d\in \Delta} (Md)%|_{[\xi_p,\xi_q]}
    \end{align}
The following lemma holds.

    \begin{lemma}
      \label{lem:inclusion_after_linear_transform}
      Formula~\eqref{eq:zero_in_hull_gradients} is equivalent to
      \begin{equation}\label{eq:simplified_inclusion}
      \zero \in \co \big\{\A_p^q,\C\big\},
        \forall \C\in \coprod_{\Delta\in \U_p^q} {\B_p^q}_\Delta.
      \end{equation}
    \end{lemma}

    \begin{proof}
      The proof is straightforward: it suffices to apply the linear
      transformation on both sides of the equation.
    \end{proof}

  \section{Identifying stationary subintervals}\label{sec:identify_stationary_subintervals}

    In this section we obtain the main results of this paper. By
    proposition~\ref{cor:confined_inclusion_equivalent_to_inf_stationarity}, we only need to find a stationary interval
    \([\xi_p,\xi_q]\).

 \subsection{Optimality conditions through block subintervals}\label{subsection:block}

  We start with the following lemma.
      \begin{lemma}
        \label{lem:stationary_subinterval}
        Given two intervals \([\xi_l,\xi_p]\) and \([\xi_q,\xi_r]\), such that
        \(p\leq q\) and \(\xi_l,\) \(\xi_p,\) \(\xi_q,\) \(\xi_r\) are not neutral, at least one of these intervals is stationary if and only if
        \begin{equation}
          \label{eq:zero_in_union_convex_hulls}
          \zero \in \co \big\{\S_l^p,\S_q^r,\C_l^p\cup\C_q^r\big\},
          \forall \C_l^p \in \coprod_{\Delta\in \U_l^p} \Delta, \forall \C_r^q \in \coprod_{\Delta\in \U_r^q} \Delta.
        \end{equation}
      \end{lemma}

      \begin{proof}
        First, since both \(\co \{\S_l^p,\C_l^p\}\) and \(\co
          \{\S_q^r,\C_q^r\}\) are subsets of
        \(\co\{\S_l^p,\S_q^r,\C_l^p\cup\C_q^r\}\),
        the stationarity of either \([\xi_l,\xi_p]\) or \([\xi_q,\xi_r]\)
        implies that inclusion.

        Now, assume that equation~\eqref{eq:zero_in_union_convex_hulls} is true
        and that the interval \([\xi_q,\xi_r]\) is not stationary. Then, we can
        find a set \(\C_q^r\in \coprod_{\Delta\in \U_q^r}\B_\Delta\)
        such that
        \begin{equation}
          \label{eq:xiqr_not_stationary}
          \zero \notin \co \big\{\A_q^r,\C_q^r\big\}.
        \end{equation}
        Applying the linear transformation to the above formula, we obtain that
        for any \(\C_l^p \in \coprod_{\Delta\in \U_l^p} {\B_p^q}_\Delta \), there exists
        \(\lambda\in [0,1]\) such that
        \[
          \begin{pmatrix}
            \zero\\\zero
          \end{pmatrix} \in
          \lambda~\co\Bigg\{
          \begin{pmatrix}
            \A_l^p\\\zero
          \end{pmatrix},
          \bigcup_{\delta\in \C_l^p}
          \begin{pmatrix}
            \delta\\\zero
          \end{pmatrix}
        \Bigg\} +
          (1-\lambda)~\co\Bigg\{
          \begin{pmatrix}
            \zero\\ \A_q^r
          \end{pmatrix},
          \bigcup_{\delta\in \C_q^r}
          \begin{pmatrix}
            \zero\\
            \delta
          \end{pmatrix}\Bigg\}.
        \]
        Formula~\eqref{eq:xiqr_not_stationary} ensures that \(\lambda \neq 0\),
        which implies that
        \( \zero \in \co \big\{\A_l^p,\C_l^p\big\} \)
        and so the interval \([\xi_l,\xi_p]\) is stationary.
      \end{proof}

      \begin{corollary}\label{cor:internal_knots_are_unstable}
        If a stationary interval \([\xi_p,\xi_q]\) contains a
        stable internal knot or an internal knot which is not an extreme point~\(\xi_r\in (\xi_p,\xi_q)\), then either
        \([\xi_p,\xi_r]\) or \([\xi_r,\xi_q]\) is stationary.
      \end{corollary}
      \begin{proof}
        If the knot \(\xi_r\) is stable or not an extreme point, then \(\S_p^q =
          \S_p^r \cup \S_r^q\) and \(\U_p^q = \U_p^r \cup \U_r^q\).
        The application of lemma~\ref{lem:stationary_subinterval} is
        straightforward.
      \end{proof}

\subsection{Characterization using alternating extreme points}\label{subsec:optim_through_alternating points}

We can now proceed with the proof of theorem~\ref{thm:main_characterization}.
\begin{proof}[Proof of theorem~\ref{thm:main_characterization}]
  We know from
  proposition~\ref{cor:confined_inclusion_equivalent_to_inf_stationarity} that
  a spline is inf-stationary if and only if there exists a stationary
  subinterval. Let \([\xi_p,\xi_q]\) be such a subinterval not containing any
  strict stationary subintervals. Equivalently, by
  proposition~\ref{prop:blocksubintervals} and applying the linear
  transformation described in section~\ref{sec:chang-vect-basis}, for any
  \(\C \in \coprod_{\Delta\in \U_p^q} \Delta\)
  \begin{equation}\label{eq:zero_conv_hull_main}
    \zero_{N(m+1)} \in \co \big\{Md: d\in \S_p^q\cup \C\big\}.
  \end{equation}
Let $\xi_{j_1},\dots,\xi_{j_{k-1}}$ be the nonsmooth knots located in $(\xi_p,\xi_q)$, that is  $\xi_{j_1},\dots,\xi_{j_{k-1}}$ separate the corresponding block subintervals.
  Noticing that the coefficients of all the vectors in the above set
  not corresponding to interval \([\xi_p,\xi_q]\) are \(0\), one can see that 
\begin{itemize}
\item $\zero$ from~(\ref{eq:zero_conv_hull_main}) can be formed without $\xi_p$ and $\xi_q$ if they are unstable;
\item for every block subinterval $[\xi_{j_i},\xi_{j_i+k_i}],$ containing $k_i$ unit subintervals
%, denoting the elements of \(\C=\{\delta_{j_1}\ldots\delta_{j_{k-1}}\}\)
%  \begin{equation}
%    \label{eq:zero_convex_hull_after_linear_transform}
%        \zero \in \co \left\{
%        \begin{pmatrix}
%          \A_p^{j_1} \\ \zero \\ \vdots \\ \zero  \\ \zero
%        \end{pmatrix}
%        \begin{pmatrix}
%          \lambda_1\delta_{j_1} \\ (1-\lambda_1)\delta_{j_1} \\ \vdots \\ \zero \\ \zero
%        \end{pmatrix}
%        \begin{pmatrix}
%          \zero \\\A_{j_1}^{j_2} \\ \vdots \\ \zero \\ \zero
%        \end{pmatrix}
%        \cdots
%        \begin{pmatrix}
%          \zero \\ \zero \\ \vdots \\ \lambda_{k-1}\delta_{j_{k-1}} \\ (1-\lambda_{k-1})\delta_{j_{k-1}} \\
%        \end{pmatrix}
%        \begin{pmatrix}
%          \zero \\ \zero \\ \vdots \\ \zero  \\ \A_p^{j_k}
%        \end{pmatrix}
%      \right\},
%    \end{equation}
%  where \(\lambda_j\in \{0,1\}\) for all \(j=1\ldots k\), where each subsystem
%  has the form 
(ignoring rows with all zeros):
  \begin{equation}\label{eq:zero_in_Tblocksimple}
    \begin{aligned}
    \zero_{mk_i+1} \in \alpha_{j_i}\co &\left\{
        \begin{matrix}
          \sigma(\xi_{j_i}) (1-\lambda_{j_i})
          \begin{pmatrix} 1\\ \eta_{j_i+1}(\xi_{j_i})\\  \zero\\ \vdots
            \\
            \zero\end{pmatrix},
        &
          \sigma(t)
          \begin{pmatrix} 1\\ \eta_{j_i+1}(t)\\  \zero\\ \vdots
            \\
            \zero\end{pmatrix},
%        &
%          \sigma(t)
%          \begin{pmatrix} 1\\ \eta_{j_i+1}(t)\\  \eta_{j_i+2}(t)\\
%            \vdots\\
%            \zero\end{pmatrix}
%        &
        \cdots
          \\
          &
          t\in \S_{j_i}^{j_i+1}
%        &
%          t\in E_{j_{i}+1}^{j_{i}+2}
      \end{matrix}\right.\\
    &\left.
      \begin{matrix}
           \sigma(t)
          \begin{pmatrix} 1\\ \eta_{j_i+1}(t)\\  \eta_{j_i+2}(t)\\ \vdots
            \\
            \eta_{j_i+k_i}(t)\end{pmatrix},
        &
           \sigma(\xi_{j_i+k_i}) \lambda_{j_i+1}
          \begin{pmatrix} 1\\ \eta_{j_i+1}(\xi_{j_i+k_i})\\  \eta_{j_i+2}(\xi_{j_i+k_i})\\ \vdots
            \\
            \eta_{i+k_i}(\xi_{j_i+k_i})\end{pmatrix}
          \\
          t\in \S_{j_{i}+k_i}^{j_{i}+k_i-1}
        \end{matrix}
      \right\}
      \end{aligned},
  \end{equation}
\end{itemize}
where $ \sigma(t)=\sign (s(t)-f(t)),$ that is the sign of the deviation function of $f(t)$ from spline $s(t)$,  $\lambda_{j_1}=1,~\lambda_{j_k}=0$.
% $E_{a}^{b}$ is the set of all extreme points of $[a,b],$ excluding neutral and unstable knots.

The system~\eqref{eq:zero_in_Tblocksimple} is equivalent to solving a linear
system involving a block triangular matrix where each block is a Vandermonde
matrix. It was studied in~\cite{tarashnin96}
(see also~\cite{sukhorukovaoptimalityfixed} where a similar system was
considered). There it is proved that the interval \([\xi_{j_i},\xi_{j_{i+1}}]\)
  contains a sequence of \(\tau_i\) unit subintervals $(\tau_i\leq k_i)$ with at least \(m\tau_i+2\)
  alternating extreme points and none of these points coincides with any of the internal knots of the block of $\tau_i$ unit subintervals.

Consider the system of all the alternating extreme points, taking the same point twice if this point appears in both adjacent block subintervals. The total number of such points is 
$$2j_k+m\sum_{i=1}^{j_k}\tau_i.$$ 

If each unstable internal knot of $[\xi_p,\xi_q]$ is represented twice in this sequence, then (since $j_k=l+1$) the total number of distinct alternating extreme points is exactly $m(q-p)+2+l,$ where $l$ is the number of non-neutral (by corollary~\ref{cor:internal_knots_are_unstable} these knots have to be unstable) internal knots inside of  $[\xi_p,\xi_q]$. In this case all the conditions of our main theorem are satisfied.

If there exists an unstable internal knot $\xi_{j_r}$ of $[\xi_p,\xi_q]$, which is not included twice, that is, not included to the subsequence of alternating extreme points for one or both adjacent block subintervals, then $\zero$ from~(\ref{eq:zero_conv_hull_main}) can  be formed on a shorter interval. Namely, if the block subinterval $[\xi_{j_{r-1}},\xi_{j_r}]$ does not include $\xi_{j_r}$ then  $\zero$ from~(\ref{eq:zero_conv_hull_main}) can  be formed by $[\xi_p,\xi_{j_r}].$ 
This proves the theorem.
\end{proof}

  \section{Discussion and conclusive remarks}\label{sec:disc-concl-remarks}

  In this paper we obtained a characterization theorem for the approximation of
  continuous functions by continuous polynomial splines with free
  knots. This result is equivalent to the Demyanov-Rubinov stationarity.

  Although our result does not address exactly the same problem as existing
  characterisation theorem for polynomial splines with free knots, several
  observations can be made:
  \begin{enumerate}
    \item In~\cite{Nurnberger892}, the number of alternating extreme points
      depends on the multiplicity of the knots. In particular, knots of
      multiplicity higher than \(2\) decrease the number of extreme points.
      This may result in suboptimal solutions, as illustrated in 
      example~\ref{ex:too_smooth}.
      In contrast, the results provided in this paper
      demonstrate that the characterization of continuous splines only depends
      on whether the spline is differentiable at the knot or not (that is its
      multiplicity is 0 or positive). This eliminates many splines, such as
      the one shown in that example.
    \item The present characterization theorem also requires the first and last
      extreme points of the sequence to not be unstable. This improvement can
      have an interesting use in an algorithm for constructing a best
      polynomial spline approximation. Indeed, as there are a variety of
      efficient algorithms for best approximation by polynomial splines with
      fixed knots, it is natural to develop algorithms alternating between
      finding the polynomial coefficients and the knots (such
      as~\cite{Meinardus89}). Our results offer an interesting strategy to
      improve on local minimisers by moving knots to coincide with the first
      and last extreme points of the stationary interval. This approach will be
      studied in details in our future work.
  \end{enumerate}

    \bibliographystyle{amsplain}
    \bibliography{julien}

\end{document}